\documentclass{article}
\usepackage{amssymb, amsmath}
\usepackage{amsthm}
\usepackage{graphicx}
\usepackage{enumerate}

\usepackage{hyperref}

\usepackage[paperwidth=169mm, paperheight=239mm]{geometry}
\theoremstyle{definition}

\newcommand{\ttt}{\mathcal{T}}
\newtheorem{theorem}{Theorem}
\newtheorem{remark}{Remark}
\newtheorem{corollary}{Corollary}

\newtheoremstyle{dotless}{}{}{\itshape}{}{\bfseries}{}{ }{}
\theoremstyle{dotless}
\newtheorem*{cndtn}{}

\newcommand{\ed}{\stackrel{D}{=}}
\newcommand{\ltwoto}{\stackrel{\lll^2}{\to}}

\newcommand{\seg}{see, e.g.,~}
\newcommand{\N}{\mathbb{N}}
\newcommand{\R}{\mathbb{R}}
\newcommand{\Z}{\mathbb{Z}}

\newcommand{\I}{\mathbf{1}}

\newcommand{\Beta}{\mathrm{B}}
\DeclareMathOperator{\E}{{\bf E}}

\renewcommand{\P}{{\bf P}}

\newcommand{\bbb}{\mathcal{B}}
\newcommand{\eee}{\mathcal{E}}

\newcommand{\iii}{\mathcal{I}}
\renewcommand{\lll}{\mathcal{L}}

\newcommand{\cond}{\hspace*{1ex} \rule[-1ex]{0.15ex}{3ex} \hspace*{1ex}}

\newcommand{\thru}{,\dotsc,}
\newcommand{\iid}{i.i.d.\ }

\newlength{\querylen}
\usepackage{fancybox}

\newcommand{\ttau}{\widetilde{\tau}}

\newcommand{\tpi}{\widetilde{\pi}}
\newcommand{\tW}{\widetilde{W}}

\newcommand{\var}{\mathbf{var}\,}
\newcommand{\cov}{\mathbf{cov}}
\newcommand{\cpi}[3]{{\pi_{#1,#2}(#3)}}
\newcommand{\ncpi}[3]{{\pi_{#1 #2}(#3)}}

\newcommand{\Dirichlet}{\mathrm{Dirichlet}}

\begin{document}

\title{Neighbour-dependent point shifts and random exchange models:
invariance and attractors}

\author{
  Anton Muratov\thanks{Chalmers
    University of Technology, Department of Mathematical Sciences,
    Gothenburg, Sweden. Email:
    \texttt{[muratov|sergei.zuyev]@chalmers.se}} 
  \and \addtocounter{footnote}{-1}Sergei Zuyev\footnotemark}
\maketitle

\begin{abstract} 
  Consider a stationary renewal point process on the real line and
  divide each of the segments it defines in a proportion given by \iid
  realisations of a fixed distribution $G$ supported by [0,1]. We ask
  ourselves for which interpoint distribution $F$ and which division
  distributions $G$, the division points is again a renewal process
  with the same $F$? An evident case is that of degenerate $F$ and
  $G$. Interestingly, the only other possibility is when $F$ is Gamma
  and $G$ is Beta with related parameters. In particular, the division
  points of a Poisson process is again Poisson, if the division
  distribution is Beta: $\Beta(r,1-r)$ for some $0<r<1$.

  We show a similar behaviour of random exchange models when a
  countable number of `agents' exchange randomly distributed parts of
  their `masses' with neighbours. More generally, a Dirichlet
  distribution arises in these models as a fixed point distribution
  preserving independence of the masses at each step. We also show
  that for each $G$ there is a unique attractor, a distribution of the
  infinite sequence of masses, which is a fixed point of the random
  exchange and to which iterations of a non-equilibrium configuration
  of masses converge weakly. In particular, iteratively applying
  $\Beta(r,1-r)$-divisions to a realisation of any renewal process
  with finite second moment of $F$ yields a Poisson process of the same
  intensity in the limit.



\end{abstract}

\section{Introduction}
\label{sec:introduction}


A Poisson point process is one of the fundamental objects in
Probability. Despite being one of the simplest to define and one of
the most-studied models, recent developments in stochastic calculus,
stochastic geometry, differential geometry of configuration spaces,
variational analysis on measures, continue to bring new insights and
deepen our understanding of this seemingly elementary concept, \seg
\cite{PecRei:15} and the references therein. It is hard to
underestimate the usefulness of the Poisson process in applications which
are due to its appealing properties often enabling to produce
mathematically tractable models. One of such fundamental properties of
a Poisson process is that whenever we apply independent random shifts
to its points, the resulting process is again Poisson. In particular,
if the shifts are i.i.d., the result of such a transformation of a
homogeneous Poisson process is again a homogeneous
Poisson process of the same intensity.

The fact that random independent shifts preserve the Poisson process
distribution reflects `independence' of its points. However,
transformations which depend on two or more neighbouring points
destroy this independence and hence, as one may naturally think, the
Poisson process. For instance, the mid-points of the consecutive
segments in a homogeneous Poisson process on the line do \emph{not}
form a Poisson process, as it can be easily checked. If one considers
again the midpoints of this process, this second iteration of the
initial process actually corresponds to the transformation by which
the points move to the centres of their one-dimensional Voronoi cells
(which are also their centres of gravity). Such a transformation,
known as the \emph{adjustment process}, can be defined in any
dimension with respect to any metric and it forms the basis of a
popular Lloyd's algorithm in computational geometry. The adjustment
process is used to model behaviour of repulsing particles or animals,
\seg \cite[Chap.~7.3.2]{OBSC:00}, Lloyd's algorithm and its variations
is widely used in image compression and optimisation, \seg
\cite{DFG:99} and the references therein. It can be shown that the
variance of the inter-point distances in one-dimensional adjustment
model on any compact set vanishes, so iterations of the adjustment
procedure converge (in a suitable sense) to a regular array of points,
see~\cite{HasTan:76}. A similar phenomenon of convergence to a lattice
is observed in a multi-dimensional case, although questions of the
uniqueness of the limiting configuration still remain,
see~\cite{DEJ:06}.


If dividing the consecutive segments in the Poisson process in half
(as well as in any other given non-random proportion) produces a
non-Poisson process, an interesting question raises: is there
a way to divide each of the segments independently of the others in a \emph{random}
proportion so that the division points still form a Poisson process?

Somewhat surprisingly, the answer is \emph{yes}: the segments must
be divided in Beta-distributed proportions with parameters
$(r,1-r)$ for some $r\in(0,1)$ and this is \emph{the only} class of
division distributions that preserves the Poisson process!


More generally, in Section~\ref{sec:2} we consider a stationary
renewal point process, the segments between its consecutive points
having lengths drawn independently from a distribution $F$. Divide its
every segment in a random proportion independently drawn from a given
distribution $G$ on $[0,1]$. One may think of the processes of the
division points as a transformation of the original point process by
which its every point moves to the closest division point on its
right, for instance.  Since the distribution of the division points
depends on the distance to the neighbouring point to the right only,
we call such transformation \emph{neighbour-dependent shifts}. We show
in Section~\ref{sec:2} that Beta-distributed shifts are the only
non-trivial ones which preserve a renewal process and that a renewal
process preserved by neighbour-dependent shifts is necessarily a
process with Gamma-distributed segments. Poisson process with
exponential segment lengths provides an example.

Since the resulting process has interpoint segments composed each of
pieces of two original segments, one may naturally consider general
compositions involving more than two of these. To this end, in
Section~\ref{sec:exchange} we establish correspondence of the
neighbour-dependent shifts model to the so-called random exchange
process which allows for such a composition interpretation. In this
process a set of `agents' simultaneously exchange their `masses' with
the neighbours in randomly drawn proportions. We show that independent
gamma distributed initial masses are preserved by
Dirichlet-distributed proportions. Moreover, we prove that the
iterations of the exchange process starting from any stationary
sequence of masses with a finite second moment weakly converge to a
limit which is a sequence of independent Gamma-distributed masses in
the case of Dirichlet-distributed proportions. To our knowledge, so
far only the exchange models with a finite number of agents were
studied in the literature, the main tool here is analysis of
convergence of the product of random matrices, see \cite{McKin:14} and
the references therein.  The r\^ole of the agents and masses in our
models play the segments and their lengths, respectively, so we have
to deal with the product of infinitely-dimensional linear operators
instead. Although the results of Section~\ref{sec:exchange} are
somewhat reminiscent of the finite case, the behaviour in this
non-compact framework is rather different.

Throughout this paper,
$\Gamma(a,\gamma)$ denotes the Gamma distribution with shape parameter
$\alpha$ and rate parameter $\gamma$, its density is given by
\begin{displaymath}
  f_\Gamma(x)=\frac{\gamma^\alpha}{\Gamma(\alpha)}x^{\alpha-1}e^{-\gamma
    x},\ x>0,
\end{displaymath}
and $\Beta(\alpha,\beta)$ is the Beta distribution with density
\begin{displaymath}
  f_\Beta(x)=\frac{\Gamma(\alpha+\beta)}{\Gamma(\alpha)\Gamma(\beta)} 
  x^{\alpha-1}(1-x)^{\beta-1}, x\in(0,1).
\end{displaymath}

More generally, a random vector $X=(X_1, \dotsc, X_r)$ with support on
a $(r-1)$-dimensional simplex
$\{(x_1,x_2,\dotsc,x_r):x_1+x_2+\dotsc+x_r=1, r\geq 2\}$ has a
Dirichlet distribution with positive real parameters $(\alpha_1,
\alpha_2,\dotsc,\alpha_r)$ if its density is given by
\begin{displaymath}
  f_X(x) = \frac{\Gamma(\sum_{i=1}^r\alpha_i)}{\prod_{i=1}^r
    \Gamma(\alpha_i)}\prod\limits_{i=1}^r x_i^{\alpha_i-1}
\end{displaymath}
We write $X\sim\Dirichlet(\alpha_1, \alpha_2, \dotsc, \alpha_r)$. 
Note that $\Beta(\alpha,\beta)$ is $\Dirichlet(\alpha,\beta)$.

If $Y_i,\ i=1\thru r$ are independent random variables $Y_i\sim
\Gamma(\alpha_i,\gamma)$ with a common $\gamma>0$, then the vector
$(Y_1/Y\thru Y_r/Y)$, where $Y=Y_1+\dotsc +Y_r$, has
$\Dirichlet(\alpha_1\thru \alpha_r)$ distribution.


\section{Invariance under neighbour-dependent shifts}
\label{sec:2}


The main object of our study in this section is a stationary point
process $T$ on the real line. Each realisation of $T$ can be
associated with a countable set of the intervals
$\iii(T)=\{I_k\}_{k\in\Z}$ between its consecutive points. Given a
realisation, we are going to introduce its transformation which
involves dividing the segments between its consecutive points in a
random proportion drawn independently from a given distribution and
then studying the distribution of the resulting process of the
division points.  We fix a \emph{division distribution} $G$ supported
by $[0,1]$ and define the following (random)
operator $\Psi_G$ acting on the set of countable subsets of $\R$
without accumulation points:
\begin{displaymath}
  \Psi_GT=T'=\cup_{I\in\iii(T)} c(I),
\end{displaymath}
where $c(I)=x+b_I(y-x)$ for an interval $I=(x,y)\in\iii(T)$ and $b_I$
is a random variable taken from the distribution $G$ independently of
anything.  Geometrically, one may think that every point $x$ of $T$ is
shifted to a new location $c(I)$, where $I$ is the interval to the
right from $x$, 
this is why we call $\Psi_G$ the
\emph{neighbour-dependent shift operator}.  Obviously, when $G$ is
concentrated on 0 or on 1, the corresponding operator $\Phi_G$
preserves the distribution of any stationary point process. 

We do not specify how the indexing of the intervals is done: the right
neighbour to $I_k$ may or may not be $I_{k+1}$. A common way to define
$I_0$ as the zero-interval, i.e.\ the one containing the origin,
introduces a size bias. In addition, the zero-interval $I'_0$ of $T'$
may be composed either of pieces of $I_0$ and $I_1$ or of $I_{-1}$ and
$I_0$. To avoid unnecessary technicalities involving either
re-indexing of $T'$ or a point-stationary indexing
of $T$, we choose to work on the Palm space instead.

It is well-known that there is one-to-one correspondence between the
distributions of a stationary point process and of a stationary
sequence $\ttt=(\tau_k)_{k\in\Z}$ of positive random variables with a
finite mean. This sequence is related to the Palm version of the point
process: under the Palm distribution there is almost surely a point
$T_0=0$ of the process at the origin, and if $T_n$ denotes the $n$th
closest process point to the origin on the positive semi-axis for
$n\geq 1$ and on the negative semi-axis for $n<0$, then
$\tau_k=T_k-T_{k-1}$, $k\in\Z$, represents the lengths of the $k$th
interpoint interval $I_k$. In the case of a renewal process, the
sequence $\ttt$ is \iid drawn from a distribution $F$ with finite
mean. Relation between the distributions of the point process and the
corresponding stationary sequence is given in general by the
Ryll-Nardzewski exchange formula, \seg \cite[Sec.~13.3]{DalVJon:08}
for details.

This allows us to define a (stochastic) \emph{exchange operator}
$\Phi_G$ acting on sequences of non-negative numbers as follows: take
an \iid sequence $\{b_k\}_{k\in\Z}$ of $G$-distributed random
variables and set
\begin{equation}\label{eq:rec}
  \Phi_G\ttt=\ttt'=(\tau'_k)_{k\in\Z},\ \text{where}\ 
  \tau'_k = (1-b_{k})\tau_{k} + b_{k+1}\tau_{k+1}.
\end{equation}

The main result of this section is the characterisation of the class
of fixed points of the exchange operator~$\Phi_G$, and thus of the
random shift operator $\Psi_G$.  
The degenerate distribution concentrated at point $x$ is denoted by $\delta_x$.

\begin{theorem}\label{th:fixpt}
  Let $\ttt$ be an \iid\ sequence of positive integrable
  random variables with the distribution $F$ corresponding to a
  stationary renewal process $T$
  and $\Phi_G$ be the exchange operator~\eqref{eq:rec}. Then
  $\Phi_G(\ttt) \ed \ttt$, and thus $\Psi_G(T) \ed T$, if and only if
  one of the following alternatives is true:
  \begin{enumerate}[(i)]
  \item $F=\Gamma(\alpha,\gamma)$ and $G=\Beta(r\alpha, (1-r)\alpha)$
    for some constants $\alpha>0$, $\gamma>0$ and $r\in(0,1)$,
  \item $F=\delta_s$ for some $s\in (0,\infty)$ and
    $G=\delta_b$ for some $b\in [0,1]$.
  \end{enumerate}
\end{theorem}

\begin{proof}





  \textit{Necessity.} Considering three consecutive elements $X,Y,Z$ in
  $\ttt$, we note that $\Phi_G(\ttt)=\ttt' \ed \ttt$ implies, in
  particular, that the two consecutive elements in $\ttt'$ they contribute
  to, should also be independent, identically distributed with $F$, which
  in turn implies the following condition:
\begin{quote}
  If $X, Y, Z$ are independent $F$-distributed random variables
  and $a, b, c$ are independent $G$-distributed, then the random
  variables
  \begin{displaymath}
    (1-a)X+bY, (1-b)Y+cZ
  \end{displaymath}
  are also independent and $F$-distributed.
\end{quote}
In terms of the Laplace transforms, for $x_1, x_2<0$ we obtain:
\begin{multline*}
   \phi_{(1-a)X+bY, (1-b)Y+cZ}(x_1,x_2)
    = \E\exp\{x_1((1-a)X+bY)\\
  \qquad\qquad\qquad\qquad+x_2((1-b)Y+cZ)\}\\
   =\E\exp\{x_1(1-a)X\}\E\exp\{x_2cZ\}\E\exp\{x_1bY+x_2(1-b)Y\}\\
   \qquad =\phi_{(1-a)X}(x_1)\phi_{cZ}(x_2)\phi_{bY,(1-b)Y}(x_1,x_2).
\end{multline*}
On the other hand, the independence of $(1-a)X+bY$ and $(1-b)Y+cZ$ implies
\begin{multline*}
  \phi_{(1-a)X+bY, (1-b)Y+cZ}(x_1,x_2) = \phi_{(1-a)X+bY}(x_1)
  \phi_{(1-b)Y+cZ}(x_2)\\
  =\phi_{(1-a)X}(x_1)\phi_{bY}(x_1)\phi_{(1-b)Y}(x_2)\phi_{cZ}(x_2),
\end{multline*}
and hence
\begin{displaymath}
  \phi_{(1-b)Y,bY}(x_1,x_2)=\phi_{(1-b)Y}(x_1)\phi_{bY}(x_2),
\end{displaymath}
so the random variables $\eta_1=bY$ and $\eta_2=(1-b)Y$ are
independent.

Let us first suppose $Y\sim F$ is degenerate. Then the only case when
$Y$ is a sum of the two independent random variables $bY$ and $(1-b)Y$
is when both of them are degenerate, too. That means, the random
variable $b\sim G$ must be degenerate, leading us to alternative~(ii).

Suppose now that $Y\sim F$ is non-degenerate. Then so are $\eta_1$ and
$\eta_2$. In that case the random variables $b=\eta_1/(\eta_1+\eta_2)$
and $Y=\eta_1+\eta_2$ can be understood as shape and size variables of
the random vector $(\eta_1,\eta_2)$. Note that $b$ and $Y$ are
independent by construction, i.e.\ the shape is independent of the
size. Therefore, the only possibility for the joint distributions of
$\eta_1, \eta_2$ is for them to be independent, Gamma-distributed with
some positive shape parameters $a_1, a_2$ and a common rate $\gamma$,
see \cite[Theorem 4]{Mos:70}. Put $r=a_1/(a_1+a_2)$ and
$\alpha=a_1+a_2$. Then $b$ becomes $\Beta(r\alpha,
(1-r)\alpha)$-distributed and $Y$ conforms to $\Gamma(\alpha,\gamma)$,
proving the alternative~(i).

\textit{Sufficiency.} The alternative~(ii) trivially leads to the invariance.

For the sufficiency in case (i) it is enough to notice that due to
\cite[Theorem 4]{Mos:70}, for every $k\in\Z$ the random variables
$\tau_k b_k$ and $\tau_k (1-b_k)$ are independent, distributed as
$\Gamma(r\alpha,\gamma)$ and $\Gamma((1-r)\alpha, \gamma)$,
respectively. Hence the random variable $\tau'_k = (1-b_k)\tau_k +
b_{k+1}\tau_{k+1}$ is again Gamma-distributed, and moreover,
$\tau'_k=(1-b_k)\tau_k + b_{k+1}\tau_{k+1}$ is independent of
$\tau'_{k+1}=(1-b_{k+1})\tau_{k+1} + b_{k+2}\tau_{k+2}$, since all
their summands are independent. Hence the
sequence  $\{\tau'_k\}_{k\in\Z}$ is again i.i.d.,
$\tau'_0\sim\Gamma(\alpha,\gamma)$, thus finishing the proof.
\end{proof}

\begin{corollary}
  Since a homogeneous Poisson process with rate $\gamma$ has
  exponential $\Gamma(1,\gamma)$ distributed interpoint distances,
  the Beta division point distribution $G=B(r,1-r)$ for some $0<r<1$ is the only
  non-degenerate distribution preserving the Poisson process.
\end{corollary}
\begin{corollary}
  Every second point in a homogeneous Poisson process form a renewal process
  with $\Gamma(2,\gamma)$-distributed interpoint distances. Thus
  a uniform division distribution which is also $G=B(1,1)$ preserves
  it. This also follows from a known elementary fact that if $X,Y$ are
  independent Exponentially-distributed random variables and $U$ is a
  uniform variable independent of them, then $U(X+Y)$ and $(1-U)(X+Y)$
  are independent Exponentially-distributed random variables.
\end{corollary}


\section{Random exchange model: fixed points and convergence}
\label{sec:exchange}

As we have shown in the previous section, Beta division distribution
$G$ defines a neighbour-dependent shift operator which preserves a
renewal process with gamma-distributed interpoint distances. The two
immediate questions arise. Suppose we start from a renewal process
realisation $\ttt$ which is \emph{not} Gamma and apply the operator
$\Phi_G$ to it iteratively. Will the iterations
$\Phi_G^{(n)}(\ttt)=\Phi_G(\Phi_G^{(n-1)}(\ttt))$ converge to a Gamma
renewal process? The answer is \emph{yes}, provided the interpoint
distances have a finite second moment. Another question: if the `new'
interpoint intervals are composed of more than two `old' ones, how
much of the previous results still hold? We show below that
Gamma-distributed renewal process appears again, but the r\^ole of the
Beta distribution now plays the Dirichlet distribution.

Recall that at each iteration of $\Phi_G$ each interpoint segment,
independently of everything else, cuts a $G$-distributed proportion of
its length and passes it to the segment to the right of it, while at
the same time receiving a portion of length from the interval to the
left of it.

If we regard the intervals as `agents', and their lengths as `masses',
then the shift procedure defined by \eqref{eq:rec} can be interpreted
as a simultaneous random exchange, where at each application of
$\Phi_G$ every agent $i$ splits its current mass $\tau_i$ into two
random pieces in proportion $b_i\,:\, 1-b_i$ and shares it between
itself and its neighbour to the right, while at the same time
receiving a piece of length of its neighbour to the left. More
generally, we define a \emph{random mass exchange} model in discrete time as
follows.

Consider a countable collection of \emph{agents}, labelled by a sequence of integers
$i\in \Z$. Each agent is supplied with a non-negative entity called
its \emph{mass}. Assume that at the beginning of step $n$ the $i$th agent
has mass $\tau^n_i$, starting from some initial mass (row-)vector
$\tau^0 = (\tau^0_i)_{i\in \Z}$ at $n=0$.  Then each agent $i$ samples
a new vector of proportions $(\pi_{i, j}(n+1))_{j \in \Z},\ \sum_j
\pi_{i,j}(n+1)=1$, and distributes all of its mass between itself and other agents
accordingly, so that agent $j$ gets a portion $\tau_i^n
\cpi{i}{j}{n+1}$ of its mass from $i$, or, in a vector form,
\begin{equation}\label{eq:randexmatrix}
  \tau^{n+1} = \tau^{n} \Pi(n+1),\ n=0,1,2,\dotsc
\end{equation}
Here
\begin{displaymath}
\Pi(n) = (\ncpi{i}{j}{n})_{i,j \in \Z}, n=1,2, \dotsc
\end{displaymath}
is a mass exchange (two-side infinite) matrix, with
proportion vectors $(\cpi{i}{j}{n})_{j \in \Z}$ as its rows.
Obviously for every $n$, $\Pi(n)$ is row-stochastic:
\begin{itemize}
\item $\ncpi{i}{j}{n} \geq 0,\ i,j \in \Z,\ n =1,2,\dotsc$
\item $\sum_{j \in \Z} \ncpi{i}{j}{n} = 1,\ i\in \Z,\ n=1,2,\dotsc$
\end{itemize}

The mass exchange model can be regarded as a discrete version of a randomised
Potlatch process first defined in~\cite{Hol:81}, \cite{Lig:81}, where instead
of having a Poisson clock at each site, all of the sites' transitions are
synchronised.

For the sequel we assume two conditions on the initial mass configuration $\tau^0$:
\phantomsection
\begin{cndtn}[A1]\label{cond:A1}
$\tau^0_i$ are non-negative \iid random variables for different $i \in \Z$ with
$\mu=\E\tau^0_i < \infty$;
\phantomsection
\end{cndtn}
\begin{cndtn}[A2]\label{cond:A2}
The second moments of initial masses are finite: $\sigma^2=\var \tau^0_i < \infty$.
\end{cndtn}

To stay in the stationary framework, we will also assume that the
random exchange model is translation invariant on $\Z$:
\phantomsection
\begin{cndtn}[B1]\label{cond:B1}
  There exists a random probability \emph{sharing distribution}
  $(\pi_j)_{j\in \Z}$ on $\Z$ such that the vectors
  $(\cpi{i}{i+j}{n})_{j\in \Z}$ for different $i \in \Z,$ and
  $n=1,2,\dotsc$ are \iid copies of $(\pi_j)_{j\in \Z}$.
\end{cndtn}
Denoting by $p_{ij} = \E\ncpi{i}{j}{n}$ and by
$p_i = \E\pi_i,\ i,j\in \Z, n=1,2,\dotsc$, this implies that the
matrix of proportions' expected values
\begin{displaymath}
P = (p_{ij})_{i,j\in \Z}, n=1,2,\dotsc
\end{displaymath}
does not depend on $n$ and the following local balance condition is
satisfied:
\begin{equation}
  \label{eq:balance}
  \E \sum_i \pi_{ij}(n)=\sum_i p_{0,j-i}=\sum_k p_k=1.
\end{equation}
for all $n\in\N$ and $j\in\Z$. Here and below the indices in the sums
run over all integers, unless specified otherwise.

We will also use the following notation for matrix products:
\begin{displaymath}
  \Pi(1:n) = \Pi(1)\dotsc\Pi(n),\ \  \Pi(n:1) = \Pi(n)\dotsc \Pi(1).
\end{displaymath}







Note that the neighbour-dependent shifts in the previous section can be
regarded as a random exchange model with the two-diagonal exchange
matrix $\Pi(n) = (\cpi{i}{j}{n})_{i,j\in\Z}$ given by
\begin{displaymath}
  \cpi{i}{j}{n} = \begin{cases}
    1-b_i(n), & i=j,\\
    b_i(n), & i=j+1,\\
    0, & \text{otherwise,}
  \end{cases}
\end{displaymath}
where $b_i(n)\sim G,\ i\in\Z,\ n=1,2,\dotsc$ are \iid

\paragraph{Random walk in random environment (RWRE)}
\label{paragraph:RWRE}
There is a correspondence between the random exchange~\eqref{eq:randexmatrix} and a certain
version of RWRE which we define here.

Assume the translation invariance condition~\nameref{cond:B1}.
Introduce $\{W^n\}_{n\geq 0}$, a random walk in a random environment
(RWRE) on $\Z$, governed by the random transition probabilities
$\pi_{i,i+j}(n)$, conditional on the environment $\eee =
\sigma(\Pi(n), n=1,2,\dotsc)$:
\begin{displaymath}
  \P\left(W^{n+1}=i+j \vert W^n=i, \eee \right) = \pi_{i,i+j}(n+1) \ed \pi_{j}
\end{displaymath}
Then its $n$-step transition matrix is given by $\Pi(1:n)$.

In this interpretation of RWRE the environment is re-sampled on every
step $n=1,2,\dotsc$.  Note that if we integrate out the environment,
then due to independence of $\Pi(n)$ we can regard a single random
walker's trajectory as a usual time-homogeneous random walk on $\Z$
with transition probabilities given by 
\begin{displaymath}
  \P(W^{n+1}=i+j|W^n=i)= \E\pi_{i,i+j}(n+1)= p_j,\ i,j \in Z,\  n=1,2,\dots
\end{displaymath}

However if there is more than one random walker, the joint dynamics
are more involved. Below we will be interested in running several copies of
$W^n$ together in the same realisation of the environment
$\{\pi_{i,i+j}(n)\}$. In particular, we are going to consider the
process $Z_n$ of the difference between two
conditionally on $\eee$ independent  copies $W^n, \tW^n$ of such a
random walk:
\begin{displaymath}
Z^n = W^n - \widetilde{W}^n,\ n=1,2,\dots
\end{displaymath}
Note that given $\eee$, $Z^n$ is not a random walk, and not even
a Markov chain, since we have to know the positions of both walkers
to determine the conditional probabilities for the next step of $Z_n$.
However, if we integrate out the environment, $Z^n$ becomes
a Markov chain on $\Z$ with the transition probabilities
\begin{equation}\label{eq:ztranprob}
\P(Z^{n+1} = i+j | Z^n = i) =\begin{cases}
\sum\limits_{j_1-j_2=j} \E \pi_{j_1} \pi_{j_2} , & i=0,\\
\sum\limits_{j_1-j_2=j} p_{j_1} p_{j_2}, & i\neq 0.
\end{cases}
\end{equation}

Having this machinery at hand, we are ready to prove the sufficient conditions
for existence and uniqueness of an equilibrium point of a random exchange
model. We use notation $\Rightarrow$ to denote the
weak convergence of a sequence of infinite random sequences, meaning the weak
convergence of all of their finite-dimensional sub-vectors.

\begin{theorem}
  Under condition~\nameref{cond:B1}, there exists a unique (up to a
  distributional copy and scaling by a constant factor) fixed point
  for dynamics~\eqref{eq:randexmatrix}, i.e.\ a sequence
  $\tau^\infty=(\tau^\infty_i)_{i\in\Z}$ of (not necessarily
  independent) random variables such that:
  \begin{enumerate}[(i)]
  \item $\tau^\infty \Pi(1) \ed \tau^\infty,$  \label{cond:fixedpoint}
  \item \label{cond:weakconv} for any $\tau^0$ satisfying
    (\nameref{cond:A1}-\nameref{cond:A2}),
    \begin{equation}
      \label{eq:wc}
      \tau^0 \Pi(1:n) \Rightarrow \tau^\infty\ \text{as $n\to\infty.$}
    \end{equation}
  \end{enumerate}
\end{theorem}


\begin{proof}
Introduce a
dual version of the process:
\begin{displaymath}
\ttau^n = \tau^0 \Pi(n:1)
\end{displaymath}
The proof is based on the distributional equality:
\begin{displaymath}
\ttau^n \ed \tau^n, n=0,1,2,\dotsc
\end{displaymath}
that holds for any fixed $n=1,2,\dots$


First, prove the theorem for the constant initial conditions $\tau^0 =
\I$, where $\I$ is the sequence of ones: $\tau^0_j =1$ for all
$j\in\Z$.  Denote by $\tpi^n_{ij}$ the elements of matrix $\Pi(n:1)$
and by $\bbb^n$ the $\sigma$-algebra generated by $\tau^0$ and
$\{\Pi(k),\ k=1,\dotsc, n\}$. We have that
\begin{align}
  \E [\ttau^{n+1}_j \cond \bbb^n] = \E \bigl[\sum_{i} \tpi_{ij}^{n+1} \cond
  \bbb^n\bigr]
  = \E \bigl[\sum_{i} \sum_{k} \pi_{ik}(n+1) \tpi_{kj}^{n} \cond \bbb^n\bigr]\nonumber\\
  = \sum_{k} \tpi_{kj}^{n} \E \sum_{i } \pi_{ik}(n+1). \label{eq:sumreverse}
\end{align}
Due to~\eqref{eq:balance}, $\E \sum_{i } \pi_{ik}(n+1)=1$ for any $k\in\Z$,
so we continue~\eqref{eq:sumreverse} with
\begin{displaymath}
  \sum_{k} \tpi_{kj}^{n} \E \sum_{i } \pi_{ik}(n+1)
  = \sum_{k} \tpi_{kj}^{n} = \ttau^{n}_j.
\end{displaymath}
Thus, for every $j$, the sequence
$\{\ttau^{n}_j\}_{n\geq 0}$ is a non-negative martingale, therefore it has an
almost sure limit, call it $\ttau^\infty_j$.
The sequence $\ttau^\infty = (\ttau^\infty_j)_{j\in \Z}$ obviously
satisfies~\eqref{cond:fixedpoint} as well as \eqref{cond:weakconv} for
the initial condition $\tau^0=\I$.
Moreover,
\begin{align*}
\var \ttau_j^n &= \var \sum_{i} \tpi_{ij}^n \leq 
\E (\sum_{i} \tpi_{ij}^n)^2 
=\E\sum_{i}(\tpi_{ij}^n)^2 + \E\sum_{k,l, k\neq l}\tpi_{kj}^n\tpi_{lj}^n
\nonumber \\
&\leq \sum_{i}\E(\tpi_{ij}^n)^2 + \sum_{k\neq l}p_k p_l
\leq \sum_{i}\E(\tpi_{ij}^n)^2 + 1. 
\end{align*}
Below in~\eqref{eq:1} we show that the first term vanishes so
for any $j$ the martingale $\{\ttau_j^n\}_{n\geq 0}$ is bounded in $\lll^2$ and hence it
converges in $\lll^2$, as well as any finite subvector of $\ttau^n$.

Now assume $\tau^0$ is general satisfying \nameref{cond:A1}, \nameref{cond:A2}.
Define
\begin{displaymath}
\Delta^n = \Delta^0 \Pi(n:1),\ \text{where } \Delta^0 = \tau^0 - \mu\I.
\end{displaymath}
Now we can rewrite $\ttau^n$ as the sum:
\begin{equation}\label{eq:ttaudecomp}
\ttau^n = \Delta^n+\mu\ttau^n.
\end{equation}
The second term converges to the limit
$\mu\ttau^\infty:=\tau^\infty$ coordinate-wise almost surely and in
$\lll^2$. Now show that each coordinate of the first term
vanishes in $\lll^2$, this will imply the desired weak convergence in
(ii).

The variance of $\Delta^n_j$ is given by the following expression:
\begin{align}\label{eq:vardelta}
  \var \Delta^n_j &= \var \sum_{i} \Delta^0_i \tpi^n_{ij} =
  \sum_{i} \var \Delta^0_i \tpi^n_{ij}+
  \sum_{i\neq k} \cov(\Delta^0_i \tpi^n_{ij},\Delta^0_k \tpi^n_{kj})\nonumber\\
  &= \sum_{i} \E(\tpi^n_{ij})^2 \E(\Delta^0_i)^2 = \sigma^2 \sum_{i}\E(\pi^n_{ij})^2
\end{align}
Here the third identity is based on the following fact that can be
checked directly: if $\xi,\eta,X,Y$ are random variables such that $(X,Y)$
is independent of $(\xi,\eta)$ and $\cov(X,Y)=0$, then
\begin{displaymath}
  \cov(\xi X, \eta Y) =\E X\,\E Y\,\cov(\xi ,\eta),
\end{displaymath}
implying all the covariance terms in~\eqref{eq:vardelta} are 0.

We now prove that $\E\sum_{i}(\pi^n_{ij})^2 \to 0, n\to\infty$ by making use of
the RWRE construction introduced earlier in this section. Let $W^n, \tW^n,
n=0,1,\dotsc$ be the two copies of conditionally independent (given
$\eee$) RWRE with transition probabilities $\Pi(n)$, and let $Z^n=W^n-\tW^n,
n=0,1,\dotsc$ with $W^0=\tW^0=j$. As already noted, $Z^n$ is
a Markov chain, with transition probabilities~\eqref{eq:ztranprob}.

The key observation is the following:
\begin{equation}\label{key-obs}
\P(Z^n=0) = \E\sum_{i}\P(W^n=\tW^n=i+j\cond\eee)
= \E\sum_{i}(\pi^n_{j, i+j})^2,
\end{equation}
where the last equality follows from the conditional independence of $W^n, \tW^n$.
Now, use the translation invariance~\nameref{cond:B1} to continue:
\begin{displaymath}
\E\sum_{i}(\pi^n_{j, i+j})^2 = \sum_{i}\E(\pi^n_{j-i, j})^2 =
\sum_{i}\E(\tpi^n_{ij})^2.
\end{displaymath}


It is easy to see that $0$ is a null state of the Markov chain
$Z^n$. Indeed, starting from $0$, $Z^n$ leaves $0$ after a
geometrically distributed with parameter
$\bigl(1-\sum_{i}\E\pi_i^2\bigr)$ number of steps, and while out of
$0$, $Z^n$ behaves as a symmetrical random walk on integers, therefore
\begin{equation}
\P(Z^n=0)=\sum_{i}\E(\tpi^n_{ij})^2\to 0,\ \ n\to\infty,\label{eq:1}
\end{equation}
hence by \eqref{eq:vardelta}, $\Delta_j^n$ tends to $0$ in $\lll^2$ for any $j$.

We have shown that both terms in the
decomposition~\eqref{eq:ttaudecomp} converge in $\lll^2$
coordinate-wise: $\ttau^n_j\ltwoto \tau_j^\infty$ for all $j\in\Z$,
which in turn implies the weak convergence $\tau^n \Rightarrow
\tau^\infty$, as well as the uniqueness of the distribution of
$\tau^\infty$, thus finishing the proof.

\end{proof}


\begin{remark}
  The finite second moment condition~\nameref{cond:A2} is essential
  for the convergence~\eqref{eq:wc} even in much more restrictive
  settings than ours as the following example shows.

  Assume that all exchange proportions are
  non-random  $\pi_{i,i+j}(n) = p_{j}$ with $p_{-1}=1-p_0=p\in(0,1)$ and $p_j=0,\
  j\in\Z\setminus\{0,-1\}$.  Then
  \begin{equation}\label{eq:eulersum}
    \tau_0^n = \sum_{j=0}^n \tau^0_j p^j(1-p)^{n-j}\binom{n}{j}
  \end{equation}
  and the (almost sure) limit of the latter expression when $n$ goes
  to infinity, if it exists, is called the \emph{Euler sum} of the
  sequence $\tau^0=(\tau^0_j)_{j\in\Z}$ with parameter~$p$.
  If $\tau^0$ is a sequence of \iid random variables,
  the almost sure limit of~\eqref{eq:eulersum} 
  exists if and only if the second moment of $\tau^0_0$ is finite, see
  \cite[Theorem 1]{BinMae:85} and the references therein.  In that
  case the limit is equal to $\mu=\E\tau^0_0$. 
\end{remark}

\begin{corollary}\label{cor:balance}
  If the vector of proportions $\pi$ has a Dirichlet distribution with
  the vector of non-negative parameters $\alpha =
  (\alpha_i)_{i\in\Z}$, $\sum_{i}\alpha_i = a < \infty$, then the
  fixed point is a vector $\tau^\infty$ of independent $\Gamma(a,
  \gamma)$-distributed random variables. In particular, if $\tau^0$
  satisfies the conditions~\nameref{cond:A1},\nameref{cond:A2}, then
  \begin{displaymath}
    \tau^0 \Pi(1:n) \Rightarrow \tau^\infty,
  \end{displaymath}
  where the components of $\tau^\infty$ are independent, with
  distribution $\Gamma(a, \gamma)$ where $\gamma =a/\mu= a/\E\tau^0_0$.
\end{corollary}
 
When $\alpha_0=(1-r)\alpha$ and $\alpha_1=r\alpha$ for some $\alpha>0$
and $r\in(0,1)$, the statement of the last corollary is the sufficiency part
of Theorem~\ref{th:fixpt}(i). It is left open whether the necessity
statement is also true for the the cases when the support of
the sharing proportion distribution $\pi$ can have more than 2
indices.

A partial answer is provided by the next theorem: in the case when
$\pi$ is exchangeable on some finite subset of $\Z$, the only scenario
for which the masses of different agents in the equilibrium are
independent is indeed when $\pi$ has a Dirichlet distribution.

\begin{theorem}\label{th:fixexch}
Assume the conditions \nameref{cond:A1}, \nameref{cond:A2},
\nameref{cond:B1} are satisfied.
Assume additionally, that the support of $\pi$ is almost surely in a
compact set: $|K| = |\{i\in\Z: \P(\pi_i>0)>0\}|<\infty$, and that
every subvector of $\pi$ is exchangeable. Then
\begin{equation}\label{eq:stationarity}
\tau'= \tau \Pi \ed \tau
\end{equation}
if and only if the components of $\tau$ are Gamma-distributed: $\tau_j
\sim \Gamma(a,\gamma)$ for some $a, \gamma>0$ and $\pi$
is Dirichlet-distributed: $(\pi_i)_{i\in K} \sim
\Dirichlet((a/|K|)_{i\in K})$.
\end{theorem}

\begin{proof}
  For the simplicity of the presentation, take $K=\{0,1,2,\dotsc,m\}$,
  $m\geq 1$.

  The `if' part follows directly from the shape vs.\ size independence
  property of Gamma random vectors, as in the proof of
  Theorem~\ref{th:fixpt}.  Now prove the `only if' part.

  First, consider the joint distribution of the two components of
  $\tau'$ which are at distance $m$: $(\tau'_0, \tau'_m)$.  By the
  invariance assumption they are independent, i.e.\  in terms of
  Laplace transforms we have:
  \begin{align}\label{eq:genfactorize1}
    \phi_{\tau'_0,\tau'_m}(x_1, x_2) =\ &
    \phi_{\tau'_0}(x_1)\phi_{\tau'_m}(x_2)\nonumber\\ 
    =\
    &\prod_{i=-m}^{0}\phi_{\tau_{i}\pi_{i0}}(x_1)
    \times\prod_{j=0}^m\phi_{\tau_{j}\pi_{jm}}(x_2).
  \end{align}
  Alternatively, we can express the Laplace transform of the
  pair $(\tau'_0, \tau'_m)$ directly, taking into account the
  independence of $\tau_{i}\pi_{ij}$ for different $i$:
  \begin{multline}\label{eq:genfactorize2}
    \phi_{\tau'_0,\tau'_m}(x_1, x_2) = \E\exp\{x_1\tau'_0 + x_2\tau'_m\} \\
    = \E\exp\Big\{x_1\sum_{i=-m}^0\tau_{i}\pi_{i0}
    + x_2\sum_{j=0}^m \tau_{j}\pi_{jm} \Big\}\\
    = \prod_{i=-m}^{-1}\E\exp\{x_1\tau_{i}\pi_{i0}\} \times
    \E\exp\{x_1\tau_{0}\pi_{00} + x_2\tau_{0}\pi_{0m}\}\times
    \prod_{j=1}^m\E\exp\{x_2\tau_{j}\pi_{jm}\}\\
    = \prod_{i=-m}^{-1}\phi_{\tau_{i}\pi_{i0}}(x_1) \times
    \phi_{\tau_0\pi_{00},\tau_{0}\pi_{0m}}(x_1,x_2)\times
    \prod_{j=1}^m\phi_{\tau_{j}\pi_{jm}}(x_2).
  \end{multline}
  Comparing \eqref{eq:genfactorize1} and
  \eqref{eq:genfactorize2}, we conclude that the two quantities
  $\tau_0\pi_{00}, \tau_0\pi_{0m}$ are independent. Since the
  distribution of the vector
  \begin{displaymath}
    (\pi_{00},\pi_{01},\dotsc,\pi_{0m})
  \end{displaymath}
  is exchangeable, the random variables $\tau_0\pi_{00},
  \tau_0\pi_{01},\dotsc, \tau_0\pi_{0m}$ are pairwise independent.

  Next, consider the joint Laplace transform of the three components
  in $\tau'$: say, $\tau'_0$, $\tau'_1$ and $\tau'_m$. We can repeat
  the previous argument to arrive at the conclusion that the joint
  Laplace transform of the three quantities $\tau_0\pi_{00},
  \tau_0\pi_{01}, \tau_0\pi_{0m}$ factorizes into the product of their
  marginal Laplace transforms, making them mutually independent. Using
  the exchangeability assumption, we conclude that the random
  variables \\$\tau_0\pi_{00}, \tau_0\pi_{01},\dotsc, \tau_0\pi_{0m}$
  are 3-independent.

  Repeating this argument $m$ times yields the joint independence of
  all the components of the random vector $(\tau_0\pi_{00},
  \tau_0\pi_{01},\dotsc, \tau_0\pi_{0m})$.  Notice that $(\pi_{00},
  \pi_{01},\dotsc, \pi_{0m})$ is its shape vector and $\tau_0$ is its
  size variable. Moreover, $\tau_0$ is independent of $(\pi_{00},
  \pi_{01},\dotsc, \pi_{0m})$ by construction, so the application of
  the shape vs.\ size characterisation of Gamma-distributed random
  vectors \cite[Theorem 4]{Mos:70} finishes the proof.
\end{proof}

\section{Open problems and generalisations}
\label{sec:open-probl-gener}

The models we have considered here admit a variety of generalisations
and raise many intriguing questions. First of all, already mentioned
extension of Theorem~\ref{th:fixexch} to a non-exchangeable sharing
distribution would give a generalisation of Theorem~\ref{th:fixpt}. We
conjecture that a non-degenerate \iid limiting sequence $\tau^\infty$
is possible only for a Dirichlet sharing distribution $\pi$, but a
counterexample may well exist. If, however, the conjecture \emph{is}
true, there are arguments that a more general statement may hold
without assuming shift-invariance of $\Pi=(\pi_{ij})$. For instance,
the agents may be indexed by another countable group such as $\Z^d$
with $d\geq2$ rather than by $\Z$. Assume there exists a sequence
$q=(q_k)_{k\in \Z}$ satisfying the balance equation $qP=q$ for the
matrix $P$ of the expectations $p_{ij}=\E\pi_{ij}$. Using the same
relations between Gamma and Dirichlet distributions, one can show that
the vector of masses $\tau^0$ with independent components $\tau_i^0$
distributed as $\Gamma(a q_i,\gamma)$ with some $a,\gamma>0$ is left
invariant by the following sharing distributions now depending on the
node $k$: $(\pi_{ki})_{i\in\Z}\sim\Dirichlet((aq_kp_{ki})_{i\in\Z})$,
$k\in\Z$. In the shift-invariant case $\pi_{ij}=\pi_{j-i}$ for all
$i,j\in\Z$, the unit sequence $\I$ satisfies the balance equation and
we obtain Corollary~\ref{cor:balance}. It is interesting if this is
the only possibility to have non-trivial independent masses as a fixed
point.

Finally, an intriguing question is, if there exists a fixed point for
multi-dimensional analog of the neighbour-dependent shifts models. For
a point process in $\R^d,\ d\geq2$ neighbouring relation can be
defined in many different ways, for instance, neighbours can be
declared the nodes connected by en edge in any stationary graph having
the process points as vertices. We already mentioned in the
Introduction the adjustment process where the nodes move to the centre
of mass of their Voronoi cells. This provides an example of a shift
depending on the Delaunay graph neighbours with a hexagon lattice
vertices in $\R^2$ left intact. Whether there are any non-degenerate point
processes preserved by these neighbouring shifts is an open
question, as well as if there are neighbour-dependent shifts of any
kind preserving a multi-dimensional Poisson process. Note in this
respect that even the balance equation is hard to
satisfy on stationary graphs in $\R^d$, the reason being typically
unbounded degree of their vertices.

\paragraph{Acknowledgement}
\label{sec:acknowledgement}
\hfill\break
The authors are thankful to Stas Volkov for many fruitful discussions
and suggestions.

\end{document}